\documentclass[a4paper,12pt]{article}

\usepackage{pstricks}
\usepackage{graphicx,psfrag}
\usepackage{amsmath}
\usepackage{amssymb}
\usepackage{amsthm}

      \newcommand {\al}   {\alpha}

              \newcommand {\ve}   {\varepsilon}

                \newcommand {\Om}  {\Omega}
      \newcommand {\pl}   {\partial}

      \newcommand {\RRR}  {{\mathbb R}}

              \newcommand {\BBB}  {{\cal B}}

     \newcommand {\beq}  {\begin{equation}}
      \newcommand {\eeq}  {\end{equation}}

    \newcommand{\figbasicconstruction}{1}
    \newcommand{\fthreepictures}{2}
    \newcommand{\figinvisible}{3}
    \newcommand{\figfamily}{4}
    \newcommand{\figtrapezia}{5}
    \newcommand{\figunfolding}{6}
    \newcommand{\figtworeflections}{7}

      \newtheorem{theor}{Theorem}

      \newtheorem{predl}{Proposition}

\author{Alena Aleksenko\thanks{Department of Mathematics, Aveiro
University, Aveiro 3810, Portugal} \and Alexander Plakhov\thanks{Aberystwyth University, Aberystwyth SY23 3BZ, UK, on leave from Department of Mathematics, University of Aveiro, Aveiro 3810-193, Portugal}
}

\title{Bodies of zero resistance and\\ bodies invisible in one direction}

\date{}

\begin{document}

\maketitle

\begin{abstract}
We consider a body in a parallel flow of non-interacting particles. The interaction of particles with the body is perfectly elastic.  We introduce the notions of a body of zero resistance, a body that leaves no trace, and an invisible body, and prove that all such bodies do exist.
\end{abstract}

\begin{quote}
\end{quote}

\begin{quote}
\end{quote}

\section{Introduction}

Consider a parallel flow of point particles falling on a body at rest. The body is a bounded connected set with piecewise smooth boundary. The particles do not mutually interact; the interaction of particles with the body is perfectly elastic. That is, each particle initially moves freely, then makes one or several (maybe none) elastic reflections from the body's surface and finally, moves freely again. It is also assumed that the initial flow density is constant and the initial velocities of particles coincide; denote the initial velocity by $v_0$. There is created the pressure force of the flow on the body; it is usually called {\it resistance}. We shall also call it {\it resistance in the direction $v_0$}. The {\it problem of minimal resistance} is concerned with minimizing the resistance in a prescribed class of bodies. There is a large literature on this problem, starting from the famous Newton's aerodynamic problem \cite{N}.

For several classes of bodies under the so-called {\it single impact assumption} (each particle makes at most one reflection), the infimum of resistance is known to be positive \cite{N},\cite{BFK},\cite{BG}-
\cite{CL2}. For some other classes, where multiple reflections are allowed, the infimum equals zero \cite{P1,P2} and cannot be attained: the resistance of any particular body is nonzero.

One of the main results of this article is the demonstration that {\it there exist bodies of zero resistance}. This means that the final velocity of almost every particle of the flow coincides with the initial one.

We say that the body {\it leaves no trace} (or is {\it trackless}) {\it in the direction} $v_0$ if it has zero resistance in this direction and, additionally, the flow density behind the body is constant and coincides with the initial one.
Further, we say that the body is {\it invisible in the direction $v_0$} if the trajectory of each particle outside a prescribed bounded set coincides with a straight line. We prove that {\it there exist bodies leaving no trace and bodies invisible in one direction}.

The paper is organized as follows. In section 2, we introduce the mathematical notation and give rigorous definitions for bodies of zero resistance, bodies that leave no trace, and bodies that are invisible in one direction. Section 3 contains the historical overview of the minimal resistance problem. In section 4, we introduce families of zero resistance bodies, trackless bodies, and invisible bodies, discuss their properties, and state some open problems.

\section{Notation and definitions}

Let $\BBB \subset \RRR^3$ be a bounded connected set with piecewise
smooth boundary, and let $v_0 \in S^2$. ($\BBB$ and $v_0$ represent
the body and the flow direction, respectively.) Consider the
billiard in $\RRR^3 \setminus \BBB$. The scattering mapping $(x, v)
\mapsto (x_\BBB^+(x, v), v_\BBB^+(x, v))$ from a full measure subset
of $\RRR^3 \times S^2$ into $\RRR^3 \times S^2$ is defined as
follows. Let the motion of a billiard particle $x(t)$,\, $v(t)$
satisfy the relations $x(t) = \left\{\! \begin{array}{ll} x + vt, &
\!\text{if } \ t < t_1\\ x^+ + v^+ t, & \!\text{if } \ t > t_2
\end{array} \right.$ and
$v(t) = \left\{\! \begin{array}{ll} v, & \!\text{if } \ t < t_1\\
v^+, & \!\text{if } \ t > t_2 \end{array} \right.$ (here $t_1$,\,
$t_2$ are a pair of real numbers depending on the particular
motion); then $x^+ =: x_\BBB^+(x, v)$,\, $v^+ =: v_\BBB^+(x, v)$.

Denote $\tilde x_\BBB^+(x, v) = x^+ - \langle x^+, v^+ \rangle v^+$
and $t_\BBB^*(x, v) = -\langle x^+, v^+ \rangle$, where $\langle
\cdot \,, \cdot \rangle$ is the scalar product; then one has $x^+ + v^+ t =
\tilde x^+ + v^+ (t - t^*)$, and $\tilde x^+$ is orthogonal to
$v^+$. We also denote by $\{ v \}^\perp$ the orthogonal complement
to the one-dimensional subspace $\{ v \}$, that is, the plane that
contains the origin and is orthogonal to $v$. \vspace{2mm}

\hspace{-6.5mm}{\bf Definition.}

{\bf D$_1$}. {\it We say that $\BBB$ has {\rm zero resistance in the
direction $v_0$} if $v_\BBB^+(x, v_0) = v_0$ for almost every $x$.}

{\bf D$_2$}. {\it We say that the body $\BBB$ {\rm leaves no trace
in the direction $v_0$} if, additionally to {\rm D}$_1$, the mapping $x \mapsto
\tilde x_\BBB^+(x, v_0)$ from a subset of $\{ v_0 \}^\perp$ into $\{
v_0 \}^\perp$ is defined almost everywhere in $\{ v_0 \}^\perp$ and
preserves the two-dimensional Lebesgue measure}.

{\bf D$_3$}. {\it We say that $\BBB$ is {\rm invisible in the
direction $v_0$} if, additionally to {\rm D}$_2$, one has $\tilde
x_\BBB^+(x, v_0) = x$.} \vspace{2mm}

The condition D$_3$ is stronger than D$_2$, and D$_2$ is stronger
than D$_1$. One easily sees that if $\BBB$ is invisible/leaves no
trace in the direction $v_0$ then the same is true in the opposite
direction $-v_0$.

This definition is interpreted as follows. Suppose that there is a
parallel flow of non-interacting particles falling on $\BBB$.
Initially, the velocity of a particle equals $-v_0$; then it makes
several reflections from $\BBB$, and finally moves freely with the
velocity $v_\BBB^+(x, v_0)$, where $x$ indicates the initial
position of the particle. One can imagine that the flow is highly
rarefied or consists of light rays. (Equivalently, one can assume
that the body translates at the velocity $-v_0$ through a highly
rarefied medium of particles at rest.) The force of pressure of the
flow on the body (or the force of resistance of the medium to the
body's motion) is proportional to $R_{v_0}(\BBB) := \int_{\{ v_0
\}^\perp} (v_0 - v_\BBB^+(x, v_0))\, dx$, where the ratio equals
the density of the flow/medium and $dx$ means the Lebesgue measure
in $\{ v_0 \}^\perp$.

In the case D$_1$ one has $R_{v_0}(\BBB) = 0$. If the body has
mirror surface then in the case D$_3$ it is invisible in the
direction $v_0$. In the case D$_2$, if the body moves through a
rarefied medium, the medium seems to be unchanged after the body has
passed: the particles behind the body (actually, in the complement
of the body's convex hull) are at rest and are distributed with the
same density.

In section 4 we give examples of a body satisfying the condition D$_1$, but not satisfying D$_2$; a body satisfying D$_2$ but not D$_3$; and a body satisfying D$_3$. That is, there exists a body of zero resistance that leaves a trace (shown on Fig.\,{\fthreepictures}a); a body leaving no trace but not invisible (Fig.\,{\fthreepictures}b and {\fthreepictures}c); and an invisible body (Fig.\,{\figinvisible}).

\section{Problems of the body of minimal resistance}

The origin of least resistance problems goes back to the book {\it Principia} (1687) by Newton. Here we give an (incomplete) overview of these problems.

Consider a three-dimensional space $\RRR^3_{x_1,x_2,x_3}$ with orthogonal coordinates $x_1$,\, $x_2$,\, $x_3$, a two-dimensional set $\Om \subset \RRR^2_{x_1,x_2}$, and a positive number $h$.
\vspace{2mm}

1. Let $\Om$ be the unit circle $x_1^2 + x_2^2 \le 1$. Consider the class of bodies $\BBB$ that are (i) bounded from
above by a function $f: \Om \to [0,\, h]$ and (ii) such that any
billiard particle in $\RRR^3 \setminus \BBB$ with the initial
velocity $v_0 = (0, 0, -1)$ makes at most one reflection from $\pl
\BBB$. The condition (i) implies that $\BBB$
contains the graph $\{ (x_1, x_2, x_3): (x_1, x_2) \in \Om, \ x_3 =
f(x_1, x_2) \}$ and is contained in the subgraph $\{ (x_1, x_2,
x_3): (x_1, x_2) \in \Om, \ x_3 \le f(x_1, x_2) \}$. The condition (ii) is called {\it  single impact assumption};
under this assumption, $R_{v_0}$ allows a comfortable analytical
functional representation: $R_{v_0}(\BBB) = \int\!\!\!\int_\Om (1 +
|\nabla f|^2)^{-1} dx_1 dx_2$. Without loss of generality
one can assume that the body is given by the relation $0 \le x_3 \le
f(x_1, x_2)$.

There have been studied the minimization problem for the absolute value
of the third component of the vector $R_{v_0}(\BBB) =
(R_{v_0}^{(1)}, R_{v_0}^{(2)}, R_{v_0}^{(3)})$ in several subclasses
of this class of bodies. These subclasses, (1a)--(1d), are defined
by additional conditions on $f$.

(1a)~ $f$ is {\it concave} and {\it radial}. This problem was considered by Newton in \cite{N}; the optimal body is indicated there without a proof.

(1b)~ $f$ is {\it concave}. Thus, the class of bodies is larger than in the case 1a. The corresponding minimization problem has been studied since 1993 (see, e.g., \cite{BK}-\cite{LO}) and is not completely solved until now. The solution is known to exist and not coincide with the Newtonian one; at any points of its surface, the gaussian curvature either equals zero or does not exist. 
The solution has been obtained numerically in \cite{LO}.

(1c)~ $f$ is {\it concave} and {\it developable} \cite{LP1}. More precisely, the level set $\{ f = h \}$ is nonempty, and the smallest concave function $\tilde f$ such that $\{ \tilde f = h \} = \{ f = h \}$ coincides with $f$. The corresponding solution is given in \cite{LP1}.

(1d)~ $f$ is arbitrary (still under the {\it single impact assumption}); this problem has been considered in \cite{CL1,CL2}.

In all these cases, the infimum of resistance is positive,
$\inf_\BBB |R^{(3)}_{v_0}(\BBB)| > 0$. \vspace{2mm}

2. Consider the class of bodies $\BBB$ that are contained in the cylinder $\Om \times [0,\, h]$ and contain a cross section $\Om \times \{ c \}$,\, $c \in [0,\, h]$. For the sake of brevity, we shall call them {\it bodies inscribed in the cylinder}. Multiple reflections are allowed. If $\Om$ is the unit circle then the infimum of resistance equals zero, $\inf_\BBB |R_{v_0}(\BBB)| = 0$ (see \cite{P2}). The infimum is not attained, that is, zero
resistance bodies do not exist. This follows from the following simple proposition.

\begin{predl}\label{predl 1}
Let $\Om$ be a convex set with nonempty interior and let $\BBB$ be a body inscribed in the cylinder $\Om \times [0,\, h]$ and such that the integral $R_{v_0}(\BBB)$ exists. Then $R_{v_0}(\BBB) \ne 0$.
\end{predl}

\begin{proof}
The integral $R_{v_0}(\BBB)$ exists, that is, the function
$v_\BBB^+(x, v_0)$ is defined for almost all $x \in \Om$ and is
measurable. Using that the particle trajectory does not intersect
the section $\Om \times \{ c \}$ and  $\Om$ is convex, one concludes
that the particle initially moves in the cylinder above this
section, then intersects the lateral surface of the cylinder and
moves freely afterwards. This implies that $v_\BBB^+(x, v_0) \ne
v_0$, hence $R_{v_0}(\BBB) \ne 0$.
\end{proof}

The assumption that $\Om$ is convex cannot be omitted; in the next section there will be provided examples of zero resistance bodies for the cases where $\Om$ is a ring or a special kind of polygon with mutually orthogonal sides (see Theorem \ref{theor 2}).

\section{Zero resistance bodies and invisible bodies}

The main result of this paper is the following theorem. Fix $v_0 \in S^2$.

\begin{theor}\label{theor 1}
There exist (a) a body that has zero resistance in the direction $v_0$ but leaves a trace; (b) a body that leaves no trace in the direction $v_0$ but is not invisible; (c) a body invisible in the direction $v_0$.
\end{theor}

\begin{proof}

(a) Consider two identical coplanar equilateral triangles $\rm ABC$ and $\rm A'B'C'$, with $\rm C$ being the midpoint of the segment $\rm A'B'$, and $\rm C'$, the midpoint of $\rm AB$. The vertical line $\rm CC'$ is parallel to $v_0$. Let $\rm A''$ ($\rm B''$) be the point of intersection of segments $\rm AC$ and $\rm A'C'$ ($\rm BC$ and $\rm B'C'$, respectively); see Fig.\,{\figbasicconstruction}.
\begin{figure}[h]
\begin{picture}(0,220)
\scalebox{0.9}{
\rput(8.7,4){
\pspolygon[linestyle=dashed,linewidth=0.3pt](-4,-3.4641)(4,-3.4641)(0,3.4641)
\pspolygon[linestyle=dashed,linewidth=0.3pt](-4,3.4641)(4,3.4641)(0,-3.4641)
 \psline[linecolor=red,arrows=->,arrowscale=1.8](-4,4.5)(-4,3.6)
\psline[linecolor=red,arrows=->,arrowscale=1.8](-4,4.5)(-4,3.4641)(1.88,0.069282)
\psline[linecolor=red,arrows=->,arrowscale=1.8](-4,3.4641)(2,0)(2,-4.25)
 \psline[linecolor=red,arrows=->,arrowscale=1.8](-2,4.5)(-2,0.15)
\psline[linecolor=red,arrows=->,arrowscale=1.8](-2,4.5)(-2,0)(4,-3.4641)(4,-4.2)
 \psline[linecolor=red,arrows=->,arrowscale=1.8](-3.5,4.5)(-3.5,2.75)
\psline[linecolor=red,arrows=->,arrowscale=1.8](-3.5,4.5)(-3.5,2.598075)(1.6,-0.34641)
\psline[linecolor=red,arrows=->,arrowscale=1.8](-3.5,2.598075)(2.5,-0.866025)(2.5,-4.25)
 \psline[arrows=<->,arrowscale=1,linewidth=0.6pt](-4,4.3)(-3.5,4.3)
 \rput(-3.73,4.52){\small $dx$}
  \psline[arrows=<->,arrowscale=1,linewidth=0.6pt](2,-3.59)(2.5,-3.59)
 \rput(2.27,-3.77){\small $dx$}
\pspolygon[fillstyle=solid,fillcolor=lightgray](-4,-3.4641)(-4,3.4641)(-2,0)
\pspolygon[fillstyle=solid,fillcolor=lightgray](4,-3.4641)(4,3.4641)(2,0)
\rput(-4.25,-3.4641){A}
\rput(-4.3,3.3){A$'$}
  \rput(-3.65,2.43){\small E}
\rput(4.27,-3.3){B}
\rput(4.25,3.464){B$'$}
\rput(0,3.75){C}
\rput(0,-3.75){C$'$}
\rput(-2.4,0){\small A$''$}
\rput(2.4,0){\small B$''$}
   \rput(2.7,-0.8){\small F}
 \psline[arrows=->,arrowscale=1.8](-5,4.5)(-5,3.46)
 \rput(-5.35,4){$v_0$}
}
}
\end{picture}
\label{fig basic construction} \caption{The basic construction.}
\end{figure}
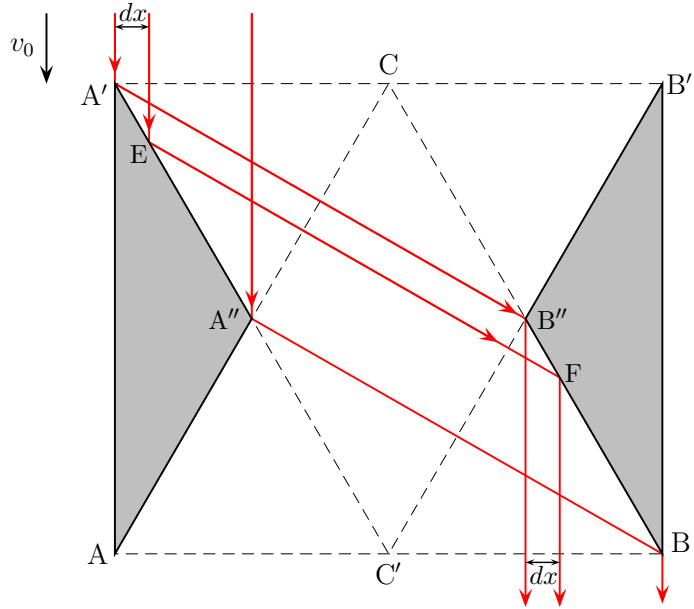
The body $\BBB$ generated by rotation of the triangle $\rm AA'A''$ (or $\rm BB'B''$) around the
axis $\rm CC'$ is shown on Fig.\,{\fthreepictures}a. It has zero resistance in the direction $v_0$. This can be better seen from Figure {\figbasicconstruction} representing a vertical central cross section of $\BBB$.

\begin{figure}[h]
\begin{picture}(0,130)
  \scalebox{0.8}{
  \rput(3.1,2.3){
   \psellipse[linewidth=1.2pt](0,-1.73205)(2,0.5)
   \psframe[linecolor=white,fillstyle=solid,fillcolor=white,linewidth=0.6pt](-2,-1.73205)(2,-1.2)
      \psellipse[linestyle=dashed,linewidth=0.6pt](0,-1.73205)(2,0.5)
         \psellipse[linewidth=1.2pt](0,1.73205)(2,0.5)
 \psline[linewidth=1.2pt](-2,-1.73205)(-2,1.73205)
 \psline[linestyle=dashed,linewidth=0.6pt](-2,-1.73205)(-1,0)
 \psline[linestyle=dashed,linewidth=0.6pt](-1,0)(-2,1.73205)
     \psellipse[linestyle=dashed,linewidth=0.6pt](0,0)(1,0.25)
  \psline[linewidth=1.2pt](2,-1.73205)(2,1.73205)
  \psline[linestyle=dashed,linewidth=0.6pt](2,-1.73205)(1,0)
  \psline[linestyle=dashed,linewidth=0.6pt](1,0)(2,1.73205)
}}
 \scalebox{0.5}{ \rput(15.6,4){
\pspolygon[fillstyle=solid,fillcolor=gray,linewidth=1.6pt](-4,-3.4641)(-4,3.4641)(-2,0)
   \pspolygon[fillstyle=solid,fillcolor=lightgray,linewidth=1.6pt](-4,3.4641)(-2,0)(-1.1,0.3)(-2.8,3.2641)
   \pspolygon[fillstyle=solid,fillcolor=lightgray,linewidth=1.6pt](-4,-3.4641)(-2,0)(-1.1,0.3)(-2.8,-2.6641)
\pspolygon[fillstyle=solid,fillcolor=gray,linewidth=1.6pt](4,-3.4641)(4,3.4641)(2,0)
   \pspolygon[fillstyle=solid,fillcolor=lightgray,linewidth=1.6pt](4,-3.4641)(4,3.4641)(4.9,3.2641)(4.9,-2.6641)
} }
 \rput(0,0.8){(a)}
  \rput(5.1,0.8){(b)}
    \rput(11.6,0.8){(c)}
   \scalebox{0.5}{
   \rput(26.8,4){
   \psframe[linewidth=0.6pt,fillstyle=solid, fillcolor=yellow](-2,-4)(-1.5,0)
   \psframe[linewidth=0.6pt,fillstyle=solid, fillcolor=yellow](0,0)(-0.5,-4)
   \psline[linewidth=2pt](-1.5,-4)(-2,-4)(-2,0)(-1.5,0)
   \psline[linewidth=2pt](-0.5,-4)(0,-4)(0,0)(-0.5,0)
      \psframe[linewidth=0.6pt,fillstyle=solid, fillcolor=green](0,-2)(4,-1.5)
      \psframe[linewidth=0.6pt,fillstyle=solid, fillcolor=green](0,0)(4,-0.5)
   \psline[linewidth=2pt](0,-1.5)(0,-2)(4,-2)(4,-1.5)
   \psline[linewidth=2pt](4,-0.5)(4,0)(0,0)(0,-0.5)
      \psframe[linewidth=0.6pt,fillstyle=solid, fillcolor=pink](2,0)(1.5,4)
      \psframe[linewidth=0.6pt,fillstyle=solid, fillcolor=pink](0,0)(0.5,4)
   \psline[linewidth=2pt](1.5,0)(2,0)(2,4)(1.5,4)
   \psline[linewidth=2pt](0.5,4)(0,4)(0,0)(0.5,0)
        \psframe[linewidth=0.6pt,fillstyle=solid, fillcolor=blue](0,2)(-4,1.5)
        \psframe[linewidth=0.6pt,fillstyle=solid, fillcolor=blue](0,0)(-4,0.5)
   \psline[linewidth=2pt](0,1.5)(0,2)(-4,2)(-4,1.5)
   \psline[linewidth=2pt](-4,0.5)(-4,0)(0,0)(0,0.5)
   }}
\end{picture}
\label{f 3 pictures} \caption{(a) A rotationally symmetric body of zero
resistance. (b) A disconnected set leaving no trace. (c) The union of 4 sets identical to the one shown on fig.\,(b), the above view. It is simply connected and leaves no trace.}
\end{figure}
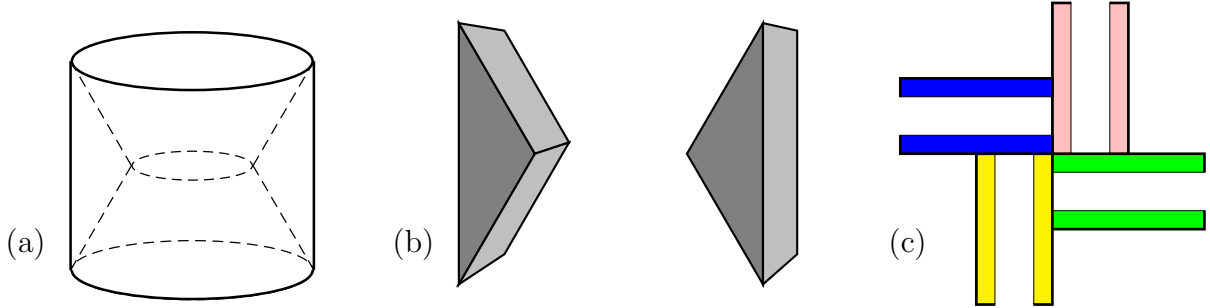

If a particle initially belongs to this cross section, it will never leave it. Let the particle first hit the segment ${\rm A}'{\rm A}''$ at a point E. (If the particle first hits ${\rm B}'{\rm B}''$, the argument is the same.) After the reflection, the direction of motion forms the angle $\pi/3$ with the vertical. Next, the particle hits the segment ${\rm B}''{\rm B}$ at the point F such that $|{\rm A}'{\rm E}| = |{\rm B}''{\rm F}|$, and after the second reflection moves vertically downward. That is, the final velocity equals $v_0$.

However, this body {\it does} leave a trace (and therefore is {\it not} invisible). Indeed, the particles that initially belong to a larger cylindrical layer of width $dx$ (on Fig.\,{\figbasicconstruction} above), after two reflections get into a smaller layer of the same width $dx$ (Fig.\,{\figbasicconstruction} below), and vice versa. Therefore, the density of the smaller layer gets larger below the body, and the density of the larger layer gets smaller. If $dx$ is small then increase and decrease of the density is twofold.

(b) A set generated by translating the pair of triangles ${\mathrm A}{\mathrm A'}{\mathrm A''}$ and ${\mathrm B}{\mathrm B'}{\mathrm B''}$ along a segment orthogonal to their plane leaves no trace in the vertical direction $v_0$, but is not invisible. It is disconnected; however, by "gluing together" 4 copies of this set along the vertical faces, one can get a {\it connected} set (that is, a true body) leaving no trace. Figure {\fthreepictures}c provides the above view of the resulting body.

(c) A body invisible in the direction $v_0$ can be obtained by doubling a zero resistance body; see Fig.\,{\figinvisible}.

Note that interior of this body is a disjoint union of two domains; this property can be undesirable. However, the construction can be improved as follows.

Consider a coordinate system $Ox_1x_2x_3$ such that the $x_3$-axis coincides with the symmetry axis of the body $\BBB$ shown on Fig.\,{\fthreepictures}a, the upper half-space contains the body, and $v_0 = (0, 0, -1)$. Consider the body $\BBB'$ symmetric to $\BBB$ with respect to the horizontal plane $x_3 = 0$ and suppose that the distance dist$(\BBB, \BBB') =: \ve$ is small. Next, take the intersection of $\BBB \cup \BBB'$ with the set $x_1 x_2 \ge 0$ (this intersection is the disjoint union of 4 connected sets) and shift it vertically up or down on $2\ve$. The union of the shifted set with the remaining set $(\BBB \cup \BBB') \cap \{ x_1 x_2 \le 0 \}$ is connected, that is, it is a true body invisible in the direction $v_0$.

\begin{figure}[h]
\begin{picture}(0,230)
\scalebox{1}{ \rput(8,2.3){
   \psellipse[linewidth=1.2pt](0,-1.73205)(2,0.5)
   \psframe[linecolor=white,fillstyle=solid,fillcolor=white,linewidth=0.6pt](-2,-1.73205)(2,-1.2)
      \psellipse[linestyle=dashed,linewidth=0.6pt](0,-1.73205)(2,0.5)
         \psellipse[linewidth=1.2pt](0,1.73205)(2,0.5)
   \psframe[linecolor=white,fillstyle=solid,fillcolor=white](-2,1.73205)(2,2.3)
      \psellipse[linestyle=dashed,linewidth=0.6pt](0,1.73205)(2,0.5)
 \psline[linewidth=1.2pt](-2,-1.73205)(-2,1.73205)
 \psline[linestyle=dashed,linewidth=0.6pt](-2,-1.73205)(-1,0)(-2,1.73205)
   \psellipse[linewidth=1.2pt](0,5.19615)(2,0.5)
     \psellipse[linestyle=dashed,linewidth=0.6pt](0,3.4641)(1,0.25)
     \psellipse[linestyle=dashed,linewidth=0.6pt](0,0)(1,0.25)
  \psline[linewidth=1.2pt](2,-1.73205)(2,1.73205)
  \psline[linestyle=dashed,linewidth=0.6pt](2,-1.73205)(1,0)(2,1.73205)
     \psline[linewidth=1.2pt](-2,1.73205)(-2,5.19615)
     \psline[linestyle=dashed,linewidth=0.6pt](-2,1.73205)(-1,3.4641)(-2,5.19615)
 \psline[linewidth=1.2pt](2,1.73205)(2,5.19615)
 \pspolygon[linestyle=dashed,linewidth=0.6pt](2,1.73205)(1,3.4641)(2,5.19615)
 \psline[linecolor=red,arrows=->,arrowscale=1.8](-1.75,6.2)(-1.75,5.3)
\psline[linecolor=red,linestyle=dashed,arrows=->,arrowscale=1.8](-1.75,5.2)(-1.75,4.763138)(0.74,3.34)
\psline[linecolor=red,linestyle=dashed,arrows=->,arrowscale=1.8](0.8,3.3)(1.25,3.0310875)(1.25,1.4)
 \psline[linecolor=red,linestyle=dashed,arrows=->,arrowscale=1.8](1.25,1.33)(1.25,0.4330125)(-0.29,-0.46)
 \psline[linecolor=red,linestyle=dashed](-0.29,-0.46)(-1.75,-1.299038)(-1.75,-1.9)
 \psline[linecolor=red,arrows=->,arrowscale=1.8](-1.75,-2)(-1.75,-2.7)
 \psline[arrows=->,arrowscale=1.8](-2.75,5.8)(-2.75,4.8391)
 \rput(-3.1,5.4){$v_0$}
} }
\end{picture}
\label{fig invisible} \caption{A body invisible in the direction $v_0$. It
is obtained by taking 4 truncated cones out of the cylinder.}
\end{figure}
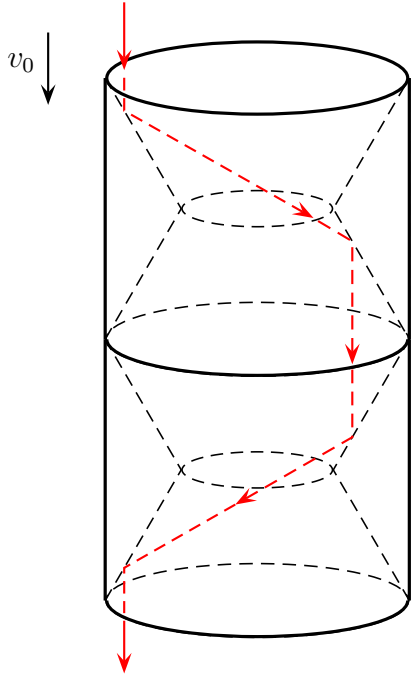

\end{proof}

Let us introduce some families of bodies having the desired properties. First, consider a pair of isosceles triangles with the angles $\al$,\, $\al$, and $\pi - 2\al$, where $0 < \al < \pi/4$. The triangles are symmetric to each other with respect to a certain point. This point lies on the symmetry axis of each triangle, at the distance $(\tan 2\al - \tan\al)/2$ from its obtuse angle and at the distance $(\tan 2\al + \tan\al)/2$ from its base. The length of the base of each triangle equals 2.
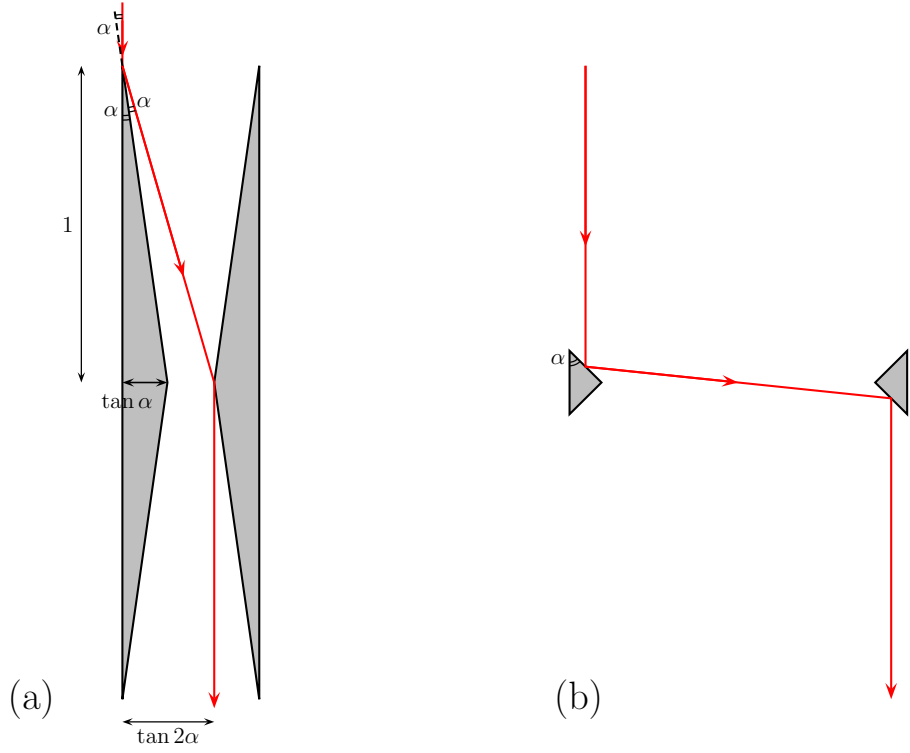
\begin{figure}[h]
\begin{picture}(0,290)
\scalebox{0.6}{
\rput(7,8){
\pspolygon[fillstyle=solid,fillcolor=lightgray,linewidth=1.3pt](-1.5,-7)(-1.5,7)(-0.513,0)
\pspolygon[fillstyle=solid,fillcolor=lightgray,linewidth=1.3pt](1.5,-7)(1.5,7)(0.513,0)
 \psline[linecolor=red,arrows=->,arrowscale=1.8,linewidth=1.3pt](-1.5,8.4)(-1.5,7.2)
\psline[linecolor=red,arrows=->,arrowscale=1.8,linewidth=1.3pt](-1.5,8)(-1.5,7)(-0.158,2.3333)
\psline[linecolor=red,arrows=->,arrowscale=1.8,linewidth=1.3pt](-1.5,7)(0.513,0)(0.513,-7.2)
\psline[linestyle=dashed,linewidth=1.3pt](-1.5,7)(-1.6691,8.2)
\psline[linewidth=1pt,arrows=<->,arrowscale=1.5](-1.5,0)(-0.513,0)
\rput(-1.4,-0.4){\large $\tan\al$}
\psline[linewidth=0.8pt,arrows=<->,arrowscale=1.5](-1.5,-7.5)(0.513,-7.5)
\rput(-0.5,-7.8){\large $\tan 2\al$}
\psline[linewidth=0.8pt,arrows=<->,arrowscale=1.5](-2.4,0)(-2.4,7)
\rput(-2.7,3.5){\large 1}
\psarc(-1.5,7){1.05}{90}{98.02}
\psarc(-1.5,7){1.13}{90}{98.02}
\rput(-1.9,7.8){\large $\al$}
\psarc(-1.5,7){1.1}{270}{278.02}
\psarc(-1.5,7){1.2}{270}{278.02}
 \psarc(-1.5,7){0.93}{278.02}{286.04}
\psarc(-1.5,7){1.03}{278.02}{286.04}
\rput(-1.75,6){\large $\al$}
\rput(-1.025,6.2){\large $\al$}
\rput(-3.5,-7){\Huge (a)}
}
\rput(19,8){
\pspolygon[fillstyle=solid,fillcolor=lightgray,linewidth=1.3pt](3,0)(3.7,0.7)(3.7,-0.7)
\pspolygon[fillstyle=solid,fillcolor=lightgray,linewidth=1.3pt](-3,0)(-3.7,0.7)(-3.7,-0.7)
 \psarc(-3.7,0.7){0.33}{-90}{-45}
  \psarc(-3.7,0.7){0.25}{-90}{-45}
  \rput(-3.95,0.53){\large $\al$}
 \psline[linecolor=red,arrows=->,arrowscale=1.8,linewidth=1.3pt](-3.35,7)(-3.35,3)
 \psline[linecolor=red,arrows=->,arrowscale=1.8,linewidth=1.3pt](-3.35,7)(-3.35,0.35)(0,0)
 \psline[linecolor=red,arrows=->,arrowscale=1.8,linewidth=1.3pt](-3.35,0.35)(3.35,-0.35)(3.35,-7)
 \rput(-3.5,-7){\Huge (b)}
}
}
\end{picture}
\label{fig family} \caption{The central vertical cross section of the body $\BBB_\al\,$ (a) with small $\al$;
(b) with $\al$ close to $\pi/4$.}
\end{figure}
On Fig.\,{\figfamily} there are depicted two pairs of triangles, with $\al$ small and $\al$ close to $\pi/4$.

As seen from the picture, this definition guarantees zero resistance in the direction $v_0$ parallel to the bases of the triangles. The zero resistance body, trackless body, and invisible body are created, respectively, by the procedures of rotation, translation with gluing, and doubling, applied to the pair of triangles.

Consider the one parameter family of zero resistance bodies $\BBB_\al$ obtained by rotation of the pair of triangles. It contains the body $\BBB = \BBB_{\pi/6}$ constructed above. Before studying the properties of this family, introduce the following definition.

For a body $\cal D$, let $\kappa(\cal D)$ be the relative volume of $\cal D$ in its convex hull, that is, $\kappa(\cal D) := \text{Vol}(\cal D)/\text{Vol}(\text{Conv}\cal D)$. One obviously has $0 < \kappa({\cal D}) \le 1$, and $\kappa({\cal D}) = 1$ {\it iff} $\cal D$ is convex.

The convex hull of $\BBB_\al$ is a cylinder of radius $L_\al = (\tan 2\al + \tan\al)/2$ and height $H = 2$; denote by $h_\al$ its relative height, $h_\al = H/L_\al$. One has Vol$(\BBB_\al) = \pi \tan\al (\tan 2\al + \tan\al/3)$. Now one easily derives the asymptotic relations for $h_\al$ and $\kappa_\al = \kappa(\BBB_\al)$: as $\al \to 0$, one has $h_\al = \frac{4}{3\al} (1 + o(1)) \to \infty$ and $\kappa_\al \to 14/27 \approx 0.52$. For $\al = \pi/6$, one has $h_{\pi/6} = \sqrt 3$ and $\kappa_{\pi/6} = 5/12 \approx 0.42$. Taking $\al = (\pi - \ve)/4$,\, $\ve \to 0^+$, one gets $h_{(\pi - \ve)/4} = 2\ve (1 + o(1))$ and $\kappa_\al = \ve (1 + o(1))$. \vspace{2mm}

Now consider a more general construction based on the union of two isosceles trapezia ${\mathrm A}{\mathrm B}{\mathrm C}{\mathrm D}$ and ${\mathrm A}'{\mathrm B}'{\mathrm C}'{\mathrm D}'$ (see
Fig.\,{\figtrapezia}).
\begin{figure}[h]
\begin{picture}(0,250)
\scalebox{1}{
\rput(7.7,4.2){
    \psline[linecolor=red,arrows=->,arrowscale=1.8](-1.176,4.8)(-1.176,4)
  \psline[linecolor=red,arrows=->,arrowscale=1.8](-1.176,4.8)(-1.176,3.61936)(0.20355,1.721)
  \psline[linecolor=red,arrows=->,arrowscale=1.8](-1.176,3.61936)(0.447,1.386)(-0.4016,1.1043)
  (0.4016,0.83735)(-0.4016,0.5704)(0.4016,0.30345)(-0.4016,0.0365)(0.4016,-0.23045)(-0.4016,-0.4974)(0.4016,-0.76435)
  (-0.4016,-1.0313)(0.42,-1.31)(-0.192,-2.154)
  \psline[linecolor=red,arrows=->,arrowscale=1.8](0.42,-1.31)(-1.11,-3.42)(-1.11,-4.4)
    \pspolygon[fillstyle=solid,fillcolor=lightgray](0.4016,1.23608)(1.29968,4)(1.29968,-4)(0.4016,-1.23608)
    \pspolygon[fillstyle=solid,fillcolor=lightgray](-0.4016,1.23608)(-1.29968,4)(-1.29968,-4)(-0.4016,-1.23608)
 \psline[linestyle=dashed,linewidth=0.3pt](0.4016,1.23608)(-0.4016,1.23608)
  \psline[linestyle=dashed,linewidth=0.3pt](0.4016,-1.23608)(-0.4016,-1.23608)
 \rput(1.55,4){\small B$'$}
  \rput(-1.5,4){\small B}
   \rput(1.55,-4){\small A$'$}
  \rput(-1.5,-4){\small A}
   \rput(0.65,1.2){\small C$'$}
  \rput(-0.65,1.2){\small C}
     \rput(0.65,-1.15){\small D$'$}
  \rput(-0.65,-1.2){\small D}
   \psarc(-1.29968,4){0.67}{270}{288}
   \psarc(-1.29968,4){0.6}{270}{288}
   \rput(-1.5,3.4){\small $\al$}
}
}
\end{picture}
\label{fig trapezia} \caption{The vertical cross section of a zero resistance body of revolution.}
\end{figure}
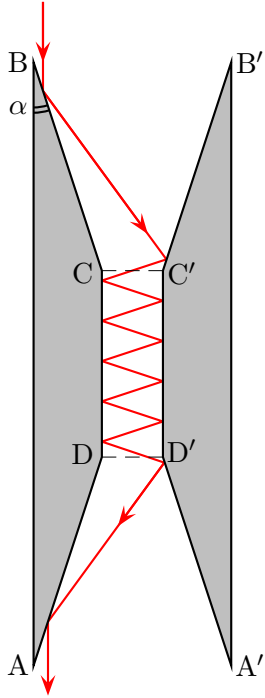
Denote by ${\mathrm E}$ the point of intersection of the lines
${\mathrm B}{\mathrm C}$ and ${\mathrm B}'{\mathrm C}'$, and take a
billiard trajectory in $\RRR^2 \setminus ({\mathrm A}{\mathrm
B}{\mathrm C}{\mathrm D} \cup {\mathrm A}'{\mathrm B}'{\mathrm
C}'{\mathrm D}')$ with the vertical initial velocity. The unfolding of the trajectory is generated
by a sequence of reflections from the lines ${\mathrm E}{\mathrm B}$
and ${\mathrm E}{\mathrm B}'$ and from their images under the
previous reflections. Let the points ${\mathrm C}$ and ${\mathrm
C}'$ be chosen in such a way that the broken line $\ldots {\mathrm
C}_2{\mathrm C}_1{\mathrm C}{\mathrm C}'{\mathrm C}'_1{\mathrm C}'_2
\ldots$ formed by the successive reflections of the segment
${\mathrm C}{\mathrm C}'$ touches the lines ${\mathrm A}{\mathrm B}$
and ${\mathrm A}'{\mathrm B}'$ (see Fig.\,{\figunfolding}a).
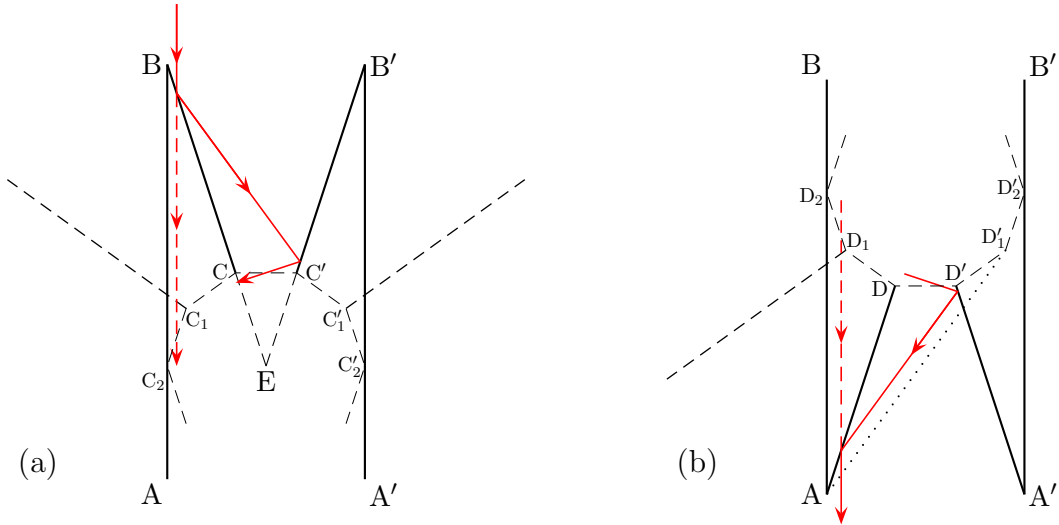
\begin{figure}[h]
\begin{picture}(0,160)
\scalebox{1}{
\rput(4,2){
\psline(0.4016,1.23608)(1.29968,4)
\psline(-0.4016,1.23608)(-1.29968,4)
  \psline[linestyle=dashed,linewidth=0.5pt](1.05145,0.76396)(3.40261,2.47215)
  \psline[linestyle=dashed,linewidth=0.5pt](-1.05145,0.76396)(-3.40261,2.47215)
     \psline[linestyle=dashed,linewidth=0.4pt](0,0)(0.4016,1.23608)
     \psline[linestyle=dashed,linewidth=0.4pt](0,0)(-0.4016,1.23608)
       \rput(0,-0.2){\small E}
    \psline[linewidth=0.8pt](1.29968,4)(1.29968,-1.5)
       \psline[linewidth=0.8pt](-1.29968,4)(-1.29968,-1.5)
    \psline[linecolor=red,linewidth=0.6pt,arrows=->,arrowscale=1.8](-1.176,4.8)(-1.176,4)
     \psline[linecolor=red,linewidth=0.6pt](-1.176,4.8)(-1.176,3.61936)
    \psline[linestyle=dashed,linecolor=red,linewidth=0.6pt,arrows=->,arrowscale=1.8](-1.176,3.61936)(-1.176,1.8)
  \psline[linecolor=red,linestyle=dashed,linewidth=0.6pt,arrows=->,arrowscale=1.8](-1.176,1.75)(-1.176,0)
         \psline[linecolor=red,linewidth=0.6pt,arrows=->,arrowscale=1.8](-1.176,3.61936)(-0.2022,2.2793)
         \psline[linecolor=red,linewidth=0.6pt,arrows=->,arrowscale=1.8](-1.176,3.61936)(0.447,1.386)
  (-0.4016,1.1043)
    \psline[linestyle=dashed,linewidth=0.3pt](-1.05145,-0.76396)(-1.29969,0)(-1.05145,0.76396)(-0.4016,1.23608)
    (0.4016,1.23608)(1.05145,0.76396)(1.29969,0)(1.05145,-0.76396)
           \rput(1.55,4){\small B$'$}
           \rput(-1.5,4){\small B}
           \rput(0.66,1.25){\scalebox{0.75}{\small C$'$}}
           \rput(-0.6,1.25){\scalebox{0.75}{\small C}}
           \rput(0.9,0.6){\scalebox{0.7}{\small C$'_1$}}
           \rput(-0.9,0.6){\scalebox{0.7}{\small C$_1$}}
          \rput(1.1,0){\scalebox{0.7}{\small C$'_2$}}
           \rput(-1.48,-0.2){\scalebox{0.7}{\small C$_2$}}
           \rput(1.55,-1.7){\small A$'$}
           \rput(-1.5,-1.7){\small A}
     \rput(-3,-1.3){(a)}
}
}
\scalebox{1}{
\rput(12.4,4.3){
\psline(0.4016,-1.23608)(1.29968,-4)
\psline(-0.4016,-1.23608)(-1.29968,-4)
  \psline[linestyle=dashed,linewidth=0.5pt](-1.05145,-0.76396)(-3.40261,-2.47215)
    \psline[linewidth=0.8pt](1.29968,-4)(1.29968,1.5)
       \psline[linewidth=0.8pt](-1.29968,-4)(-1.29968,1.5)
   %
  \psline[linecolor=red,linewidth=0.6pt,arrows=->,arrowscale=1.8]
  (-0.28,-1.075)(0.42,-1.31)(-0.192,-2.154)
  \psline[linecolor=red,linewidth=0.6pt,arrows=->,arrowscale=1.8](0.42,-1.31)(-1.11,-3.42)(-1.11,-4.4)
  \psline[linecolor=red,linestyle=dashed,linewidth=0.6pt,arrows=->,arrowscale=1.8](-1.11,-0.1)(-1.11,-2)
  \psline[linecolor=red,linestyle=dashed,linewidth=0.6pt](-1.11,-2)(-1.11,-3.42)
    \psline[linestyle=dashed,linewidth=0.3pt](-1.05145,0.76396)
    (-1.29969,0)(-1.05145,-0.76396)(-0.4016,-1.23608)(0.4016,-1.23608)(1.05145,-0.76396)(1.29969,0)(1.05145,0.76396)
      \psline[linestyle=dotted,linewidth=0.8pt](-1.29968,-4)(1.05145,-0.76396)
           \rput(1.55,-4){\small A$'$}
           \rput(-1.5,-4){\small A}
           \rput(0.4,-1.05){\scalebox{0.75}{\small D$'$}}
           \rput(-0.6,-1.28){\scalebox{0.75}{\small D}}
           \rput(0.9,-0.6){\scalebox{0.7}{\small D$'_1$}}
           \rput(-0.88,-0.65){\scalebox{0.7}{\small D$_1$}}
          \rput(1.1,0.05){\scalebox{0.7}{\small D$'_2$}}
           \rput(-1.5,-0.05){\scalebox{0.7}{\small D$_2$}}
           \rput(1.55,1.7){\small B$'$}
           \rput(-1.5,1.7){\small B}
     \rput(-3,-3.6){(b)}
}
}
\end{picture}
\label{fig unfolding} \caption{Unfolding of a billiard trajectory.}
\end{figure}
The points ${\mathrm C}$ and ${\mathrm
C}'$ are uniquely determined by this condition. This choice guarantees that the unfolded trajectory of each particle intersects this broken line; hence the original billiard trajectory, after several reflections from ${\mathrm B}{\mathrm C}$ and ${\mathrm B}'{\mathrm C}'$, eventually intersects the segment ${\mathrm C}{\mathrm C}'$.

Denote $\measuredangle{\mathrm A}{\mathrm B}{\mathrm C} =: \al$ (and therefore $\al = \measuredangle{\mathrm B}{\mathrm A}{\mathrm D} = \measuredangle{\mathrm A}'{\mathrm B}'{\mathrm C}' =  \measuredangle{\mathrm B}'{\mathrm A}'{\mathrm D}'$); we assume that $\al < \pi/4$. After the first reflection the particle velocity forms the angle $2\al$ with the vertical direction $(0, -1)$; after the second reflection the angle is $4\al$, and so on. At the point of intersection with ${\mathrm C}{\mathrm C}'$ the angle is $2k\al$, where $k$ is a positive integer such that $2k\al < \pi/2$.

While the particle belongs to the rectangle ${\mathrm C}{\mathrm C}'{\mathrm D}'{\mathrm D}$, the angle remains equal to $2k\al$, and when the particle makes reflections from the sides ${\mathrm A}{\mathrm D}$ and ${\mathrm A}'{\mathrm D}'$, the angle decreases, taking successively the values $2(k - 1)\al$,\, $2(k - 2)\al, \ldots$, and finally, after the last reflection, it becomes $2k'\al$, where $k'$ is a nonnegative integer, $0 \le k' \le k$.

Let us show that $k' = 0$ and therefore, the final velocity is vertical. To that end, let us unfold the final part of the trajectory (below the line ${\mathrm D}{\mathrm D}'$); see Fig.\,{\figunfolding}b. The broken line $\ldots {\mathrm D}_2{\mathrm D}_1{\mathrm D}{\mathrm D}'{\mathrm D}'_1{\mathrm D}'_2 \ldots$ generated by this
unfolding touches the lines ${\mathrm A}{\mathrm B}$ and ${\mathrm A}'{\mathrm B}'$ and intersects the unfolded trajectory. We see that the tangent lines drawn from ${\mathrm A}$ to this broken line (the lines ${\mathrm A}{\mathrm D}_2$ and ${\mathrm A}{\mathrm D}'_1$ on Fig.\,{\figunfolding}b) form the angles 0 and $2\al$ with the vertical; this implies that $2\al > 2k'\al$ and therefore, $k' = 0$.

The less $\al$, the less is the quotient $r := |{\mathrm C}{\mathrm C}'|/|{\mathrm B}{\mathrm B}'|$. Actually, $r = r(\al)$ is a monotone increasing continuous function varying from $r(0) = 0$ to $r(\pi/4) = 1$. The exact formula is: $r(\al) = \sin\al/\sin(2\lfloor \pi/(4\al) \rfloor \al + \al)$.

The body of zero resistance is formed by rotation of the trapezia around the vertical symmetry axis. Its shape is determined by the two parameters $\al$ and $k = |{\mathrm C}{\mathrm D}|/|{\mathrm B}{\mathrm C}|$. As $\al \to 0$ and $k \to \infty$, the maximal number of reflections goes to infinity, the relative volume of the body in the cylinder ${\mathrm A}{\mathrm B}{\mathrm B}'{\mathrm A}'$ goes to 1, and the relative height of the cylinder goes to infinity.

By doubling this body, one obtains the body invisible in the direction $v_0$.
\vspace{2mm}

This result can be summarized as follows.

\begin{theor}\label{theor 2} Let $\Om$ be a ring $r^2 \le x_1^2 + x_2^2 \le 1$. For $h$ sufficiently large, there exists a body inscribed in $\Om \times [0,\, h]$ and invisible in the direction $v_0 = (0, 0, -1)$.
\end{theor}

\hspace*{-6.5mm}{\bf Remark.} {\it This theorem is also true for the case where $\Om$ is a special kind of polygon with mutually orthogonal sides; see, e.g., Fig.\,{\fthreepictures}c.}
 \vspace{2mm}

Denote by $m = m(\BBB, v_0)$ the maximal number of reflections of an
individual particle from the body.

\begin{predl}\label{predl 2}
(a) If the body $\BBB$ has zero resistance or leaves no trace in the direction $v_0$ then $m(\BBB, v_0) \ge 2$. (b) If $\BBB$ is invisible in the direction $v_0$ then $m(\BBB, v_0) \ge 4$. These inequalities are sharp: there exist zero resistance bodies and trackless bodies with exactly 2 reflections, and there exist invisible bodies with exactly 4 reflections.
\end{predl}

\begin{proof}
(a) If $m = 1$ (that is, under the single impact assumption) then the final velocity of each particle does not coincide with the initial one, $v^+_\BBB(x, v_0) \ne v_0$, therefore $R_{v_0}(\BBB) \ne 0$. That is, a zero resistance body requires at least two reflections.

(b) Note that a thin parallel beam of particles changes the orientation under each reflection. To be more precise, let $x(t) = x + v_0 t$,\, $v(t) = v_0$ be the initial motion of a particle, and let $x(t) = x^{(i)}(x) + v^{(i)}(x) t$,\, $v(t) = v^{(i)}(x)$ be its motion between the $i$th and $(i+1)$th reflections, $i = 0,\, 1, \ldots, m$. Let the body be invisible in the direction $v_0$; then one has $v^{(0)} = v^{(m)} = v_0$,\, $x^{(0)} = x$, and $x^{(m)} - x \perp v_0$. At each reflection and for any fixed $x$, the orientation of the triple $\big( \frac{\pl x^{(i)}}{\pl x_1},\,  \frac{\pl x^{(i)}}{\pl x_2},\, v^{(i)} \big)$ changes. The initial and final orientations, $\big( \frac{\pl x^{(0)}}{\pl x_1},\,  \frac{\pl x^{(0)}}{\pl x_2},\, v^{(0)} \big)$ and $\big( \frac{\pl x^{(m)}}{\pl x_1},\,  \frac{\pl x^{(m)}}{\pl x_2},\, v^{(m)} \big)$, coincide, therefore $m$ is even.

On the other hand, $m$ cannot be equal to 2, as seen from Fig.\,{\figtworeflections}. Therefore, $m \ge 4$.
\begin{figure}[h]
\begin{picture}(0,90)
\scalebox{1}{ \rput(7.7,-1.25){
\psline[linecolor=red,arrows=->,arrowscale=1.8](0,4)(0,3.3)
\psline[linecolor=red,arrows=->,arrowscale=1.8](0,4)(0,2.9)(1,2.4)
\psline[linecolor=red,arrows=->,arrowscale=1.8](0,2.9)(1.6,2.1)(1.6,1)
 \psdots(0,2.9)(1.6,2.1)
 \rput(-0.3,2.9){1}
  \rput(1.9,2.1){2}
  } }
\end{picture}
\label{fig 2 reflections} \caption{Two reflections are not enough for an invisible body.}
\end{figure}
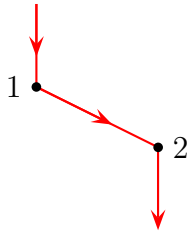

>From the examples of bodies discussed above one concludes that the inequalities in (a) and (b) are sharp.
\end{proof}

Finally, put some open questions.

\begin{enumerate}

\item Do there exist bodies invisible in more than one direction? The same question concerns bodies of zero resistance/leaving no trace.

\item For which domains $\Om$ (others than a ring) is Theorem \ref{theor 2} true?

\item The resistance of any convex body is nonzero. However, by taking a small portion of volume out of a convex body, one can get a body of zero resistance. Namely, there exists a sequence of zero resistance bodies $\BBB_n$ such that their relative volumes $\kappa(\BBB_n)$ go to 1, $\lim_{n\to\infty} \kappa(\BBB_n) = 1$. The maximal number of reflections for these bodies goes to infinity, $\lim_{n\to\infty} m(\BBB_n, v_0) = \infty$.
    The question is: estimate the maximal relative volume of a zero resistance body $\BBB$, given that the maximal number of reflections does not exceed a fixed value $m \ge 2$. In other words, estimate $\kappa_m := \sup \{ \kappa(\BBB) : R_{v_0}(\BBB) = 0, \ m(\BBB, v_0) \le m \}$. It is already known that $\kappa_m \ge 14/27$ and $\lim_{m\to\infty} \kappa_m = 1$.

\end{enumerate}

\section*{Acknowledgements}

This work was supported by {\it Centre for Research on Optimization
and Control} (CEOC) from the ''{\it Funda\c{c}\~{a}o para a
Ci\^{e}ncia e a Tecnologia}'' (FCT), cofinanced by the European
Community Fund FEDER/POCTI, and by FCT: research project
PTDC/MAT/72840/2006.


\end{document}